\documentclass{amsart}
\usepackage[all]{xy}
\usepackage{amsmath}
\usepackage{amssymb}
\usepackage{amsthm}
\usepackage{amscd}

\newcommand{\M}{\mathcal{M}}
\newcommand{\V}{\mathcal{V}}

\newcommand{\F}{\mathcal{F}}

\newcommand{\s}{\mathcal{S}}

\newcommand{\Z}{\Bbb Z}

\newcommand{\RR}{\Bbb R}

\newcommand{\C}{\Bbb C}

\newtheorem{theorem}{Theorem}

\newtheorem{proposition}[theorem]{Proposition}
\newtheorem{cor}[theorem]{Corollary}
\newtheorem{lemma}[theorem]{Lemma}
\newtheorem{problem}[theorem]{Problem}
\newtheorem{remark}[theorem]{Remark}

\DeclareMathOperator{\supp}{supp}

\numberwithin{theorem}{section}

\begin{document}
\title[Some Optimizations]{Some Optimizations for (Maximal) Multipliers in $L^p$}
\author{Ben Krause}
\address{UCLA Math Sciences Building\\
         Los Angeles,CA 90095-1555}
\email{benkrause23@math.ucla.edu}
\date{\today}
\maketitle

\begin{abstract}
Using the Multi-Frequency Calder\'{o}n-Zygmund decomposition of \cite{N1}, and a discretization suggested in the arguments of \cite{D} we lower estimates of \cite{N1} and \cite{D}.

We also use these techniques to improve the results of \cite{CRS}.
\end{abstract}

\section{Introduction}
In his celebrated paper \cite{B1}, Bourgain proved the following:

\begin{theorem}[Lemma 4.11 of \cite{B1}]
For any collection of $N$ frequencies, $\Sigma := \{\xi_1,\dots, \xi_N \} \subset \RR$, then
\[ \left\| \sup_j | \left( \hat{f} 1_{R_j} \right)^{\vee}(x) \right\|_{L^2(\RR)} \lesssim \log^2 N \|f\|_{L^2(\RR)},\]
where $R_j$ is the $2^{-j}$ neighborhood of $\Sigma$. (We refer the reader to \S 1.2 for the precise definition of the $\lesssim$ notation)
\end{theorem}

This beautiful lemma was proven via a robust argument which combined Fourier analysis and a metric entropy approach; it proved a key ingredient in his treatment of pointwise polynomial ergodic theorems. We further remark that this result is essentially sharp \cite{B2}, since it has been shown that at least a power of $\log^{1/4}N$ is needed on the right-hand side.

In recent years, this result has been generalized in a number of ways. In \cite[Theorem 8.7]{D1} a weighted $(L^2$-) version of the above result was used to extend Bourgain's Return Times theorem. In \cite[Theorem 1.5]{D}, the weighted result was extended to the lower $L^p$ regime, $1<p \leq 2$ by using elegant time-frequency techniques. The boundary $p=1$ case remained out of reach, however, until Nazarov, Oberlin and Thiele developed a Multi-Frequency Calder\'{o}n Zygmund decomposition \cite{N1}. There, they further extended the $L^2$-weighted result of \cite{D1} to the variational setting \cite[Theorem 1.2]{N1}, and additionally studied the $L^1$ endpoint \cite[Theorem 5.1]{N1}.

In this note, we improve the main results of the above papers \cite[Theorem 1.5]{D}, \cite[Theorem 1.2]{N1} (see below for the precise statements). We do so by establishing an optimal estimate for some metric entropy calculuations (see Lemma \ref{Tech} for a precise statement) and through more delicate use of the Multi-Frequency Calder\'{o}n Zygmund decomposition (see \S 3).

Our main $L^2$-result (\S 2) is the following

\begin{proposition}\label{main}
Suppose that $\Sigma:=\{\xi_1,\dots,\xi_N\} \subset \RR$ are 1-separated frequencies, and
define
\[ \aligned
D_k (f) &:= \sum_{k=1}^N \int \hat{f}(\xi) \widehat{\phi_{2^k}}(\xi-\xi_j) e(\xi x) \ d\xi  \\
&= \sum_{k=1}^N e(\xi_j x) \int \hat{f}(\xi+\xi_j) \widehat{\phi_{2^k}}(\xi) e(\xi x) \ d\xi, \endaligned\]
where we abbreviate $e(t):= e^{2\pi i t}$.
Then, for $q > 2$,
\[ \left\| \V^q(D_k (f) \ \big| k \geq 1 ) \right\|_{L^2_x(\RR)} \lesssim (\log N)^2 \cdot \left( 1 + \frac{q}{q-2} \right)^2 \cdot \|f\|_{L^2_x(\RR)},\]
where the \emph{q-variation operator}, $\V^q$, is defined below.
\end{proposition}

We define the (non-homogeneous) $q$-variation of the sequence of functions $\{ D_k (f) \}$
\[ \V^q(D_k(f))(x):= \sup_{k} |D_k(f)| +
\left( \sup_{k_1 < k_2 < \dots < k_M} \sum_m | D_{k_m} f  - D_{k_{m+1}} f |^q \right)^{1/q}(x), \]
where the supremum is taken over finite increasing sequences;
we refer the reader below for a more involved discussion of the $q$-variation norm.

We also include a simplified -- and optimized -- proof of \cite[Theorem 5.1]{N1}; this proof is easily adapted to study the $L^1$ endpoint treated in \cite{D}. We prove the following

\begin{proposition}
Suppose $\Sigma:=\{ \xi_1, \dots, \xi_n\} \subset \RR$ are not necessarily $1$-separated, and define $D_k(f)$ as above. Then, for $q > 2$,
\[ \left\| \V^q(D_k (f)) \right\|_{L^{1,\infty}(\RR)} \lesssim \sqrt{N} \cdot \log^3 N \cdot \left( 1 + \frac{q}{q-2} \right)^2 \|f \|_{L^1(\RR)}.\]
\end{proposition}
\begin{remark}
The prefactor $\sqrt{N}$ is sharp, and though suggested in the work of \cite[Theorem 5.1]{N1}, was (narrowly) missed there.
\end{remark}

The techniques used in the proof of the above proposition robust enough to prove useful in studying the behavior of Fourier multipliers with bounded $r$-variation. We briefly discuss our main result in this direction (see \S 4 for greater depth):

We define the (non-homogeneous) $r$-variation norm, the $\V^r$-norm, of a function
\[ \| h \|_{\V^r(\RR)} := \| h \|_{L^\infty(\RR)} + \left( \sup_{\xi_1 < \xi_2 < \dots < \xi_K} \sum_k | h(\xi_k) - h(\xi_{k+1}) |^r \right)^{1/r},\]
where the supremum is over all strictly increasing finite sequences of real numbers.

For $X$ a finite collection of disjoint subintervals of $\RR$, and $\{ m_\omega : \omega \in X\}$ a collection of functions with $\hat{m}_\omega$ supported in $\omega$, and $\| \hat{m}_\omega \|_{\V^r}$ uniformly bounded in $\omega$, we consider operators of the form
\[ \M f:= \left( \sum_{\omega} \hat{m}_\omega \hat{f} \right)^{\vee}.\]

In \cite{CRS}, it was proven that for any $r \geq 2$, $|\frac{1}{q} - \frac{1}{2}| < \frac{1}{r}$, and $\epsilon > 0$, there exist absolute constants $C_{q,r,\epsilon}$ so that
\[ \left\| \M f \right\|_{L^q(\RR)} \leq C_{q,r,\epsilon} | X| ^{|\frac{1}{q} - \frac{1}{2}| + \epsilon} \sup_{\omega} \| \hat{m}_\omega \|_{\V^r(\RR)} \|f\|_{L^q(\RR)}.\]

We prove that the $\epsilon$ is extraneous:
\begin{theorem}\label{rvar}
There exist absolute constants $C_{r,q}$ so that
\[ \left\| \M f \right\|_{L^q(\RR)} \leq C_{q,r} | X| ^{|\frac{1}{q} - \frac{1}{2}|} \sup_{\omega} \| \hat{m}_\omega \|_{\V^r(\RR)} \|f\|_{L^q(\RR)}.\]
\end{theorem}

We prove this result by studying the $L^1$-endpoint behavior of simpler multiplier operators (\S 4). Our approach, however, just fails to shed light on the endpoint question raised in \cite{TW}:

\begin{problem}
For $\| \hat{m}_\omega \|_{\V^2(\RR)} \leq 1$ uniformly bounded, what is the best $|X|$-dependent bound, $C_1(|X|)$, in the weak-type inequality
\[ \left\| \M f \right\|_{1,\infty} \leq C_1(|X|) \|f\|_1?\]
\end{problem}
We look forward to addressing this problem in further research.

\bigskip

The structure of the paper is as follows:

Below, we present our improvements to \cite{N1} and \cite{D}

In \S 2, we develop our $L^2$-theory;

In \S 3, we turn to the end-point $L^1$ question; and

In \S 4, we discuss multipliers with bounded $r$-variation.

\subsection{Acknowledgements}
The author would like to thank Christoph Thiele for helpful conversations, and his advisor, Terence Tao, for his support.

\subsection{Notation}
Throughout, we will make use of the exponential notation
\[  e(t ) := e^{2\pi i t}. \]

For an interval $I$, we let $CI$ denote the concentric interval dilated by a factor of $C$:
\[ I = (-l+c, c+l) \Rightarrow CI = (-Cl + c, c + Cl ).\]

In the first sections, with $q >2$ fixed, we also shall assume $N = O_q(1)$ is large -- say $\log N \frac{q}{q-2} \geq 10$; in what follows, the modifications necessary to lift this assumption are straightforward.

We will also make use of the modified Vinogradov notation. We use $X \lesssim Y$, or $Y \gtrsim X$ to denote the estimate $X \leq CY$ for an absolute constant $C$. If we need $C$ to depend on a parameter, we shall indicate this by subscripts, thus for instance $X \lesssim_r Y$ denotes
the estimate $X \leq C_r Y$ for some $C_r$ depending on $r$. We use $X \approx Y$ as
shorthand for $X \lesssim Y \lesssim X$.

\subsection{Preliminaries}
Begin by fixing a band-limited Schwartz function $\phi \in \s(\RR)$, with $\supp \hat{\phi} \subset [-1/2,1/2]$, and let $\phi_t(x) := \frac{1}{t} \phi(x/t)$ denote the usual $L^1$-normalized dilation.

We recall two families of operators which measure the fluctuation of a sequence of vectors:

For a collection of vectors in a Hilbert space $\{c_k\} \subset H$, let $M_\lambda(\{c_k\})$ denote the $H$-$\lambda$-entropy of the $\{c_k\}$, i.e.\ the fewest number of $\lambda$-balls inside of $H$ required to cover the $\{c_k\}$.

We also introduce the homogeneous and non-homogeneous variation operators, respectively defined below:
\[ \aligned
\tilde{\V}^q_H(\{c_k\}) &:= \sup_{M, k_0< k_1< k_2 < \dots<k_M} \left( \sum_m \|c_{k_m} - c_{k_{m-1}}\|_H^q \right)^{1/q},\\
\V^q_H(\{c_k\}) &:= \sup_k \|c_k\|_H + \sup_{M, k_0< k_1< k_2 < \dots<k_M} \left( \sum_m \|c_{k_m} - c_{k_{m-1}}\|_H^q \right)^{1/q}. \endaligned \]
When $H = \C$, we shall omit the subscript.
We note that for any $\lambda > 0$,
\[ \lambda M_\lambda(\{c_k\})^{1/q} \lesssim \tilde{\V}^q_H(\{c_k\}).\]

We will prove the following
\begin{proposition}\label{main}
Suppose that $\Sigma:=\{\xi_1,\dots,\xi_N\} \subset \RR$ are 1-separated frequencies, and
define
\[ \aligned
D_k (f) &:= \sum_{k=1}^N \int \hat{f}(\xi) \widehat{\phi_{2^k}}(\xi-\xi_j) e(\xi x) \ d\xi  \\
&= \sum_{k=1}^N e(\xi_j x) \int \hat{f}(\xi+\xi_j) \widehat{\phi_{2^k}}(\xi) e(\xi x) \ d\xi. \endaligned\]
Then, with $q>2$,
\[ \left\| \V^q(D_k (f) \ \big| k \geq 1 ) \right\|_{L^2_x(\RR)} \lesssim (\log N)^2 \cdot \left( 1 + \frac{q}{q-2} \right)^2 \cdot \|f\|_{L^2_x(\RR)}.\]
\end{proposition}
By using a Rademacher-Menshov style argument used in \cite[Lemma 4.32]{B1} or \cite[Theorem 4.3]{N1}, one may lift the frequency-separation hypothesis at the cost of a further logarithm:
\begin{proposition}\label{main2}
For any frequencies $\Sigma:= \{\xi_1,\dots,\xi_N\}$ and any dyadic (frequency) interval $\omega$, let $\xi_\omega$ denote the smallest frequency $\xi_j \in \Sigma \cap \omega$ which lies in $\omega$, and let
\[ \widehat{\phi_\omega}(\xi):= \hat{\phi}\left(\frac{\xi - \xi_\omega}{|\omega|} \right).\]
In this instance, define
\[ D_k (f) := \sum_{|\omega|= 2^{-k} : \ \omega \cap \Sigma \neq \emptyset } \int \hat{f}(\xi) \hat{\phi_\omega}(\xi) e(\xi x) \ d\xi . \]
Then, with $q>2$,
\[ \left\| \V^q(D_k (f) ) \right\|_{L^2_x(\RR)} \lesssim (\log N)^3 \cdot \left( 1 + \frac{q}{q-2} \right)^2 \cdot \|f\|_{L^2_x(\RR)}.\]
\end{proposition}
For the sake of clarity we have contented ourselves with the above results;
 it should be noted, however, that by arguing as in \cite[\S 4.1-4.2]{N1}, an improvement of \cite[Proposition 4.2]{N1} can be achieved:

If
\[ \Delta_k(f) := \sum_{|\omega|= 2^{-k} : \ \omega \cap \Sigma \neq \emptyset} f * m_\omega, \]
where $\{ m_\omega \}$ are Schwartz functions with
\[ \widehat{m}_\omega \subset \omega, \]
then we have the following

\begin{cor}
For any choice of $N$ frequencies, $\Sigma:=\{\xi_1,\dots,\xi_N\} \subset \RR$, there exists a $t=t(q) > 2$ so that
\[ \left\| \V^q \left( \Delta_k(f) \right) \right\|_{L^2_x(\RR)} \lesssim (\log N)^3 \cdot \left( 1 + \frac{q}{q-2} \right)^2\left( D'_2 + V^t \right) \|f\|_{L^2(\RR)},\]
where $V^t:= \sup_j \V^t \left( \sum_{|\omega|=2^{-k}} \widehat{m}_\omega(\xi_j) \right)$, and where $D'_2 := \sup_{\omega, \xi} |\omega|^2 | \hat{m}_\omega^{(2)}(\xi)|$.
\end{cor}

We also consider the $L^1$-endpoint behavior of the (non-frequency separated) operator $\V^q(D_k (f))$, $q>2$. We prove
\begin{proposition}\label{1-1}
Under the non-frequency separated hypotheses of Proposition \ref{main2}, with $q > 2$
\[ \| \V^q(D_k (f)) \|_{L^{1,\infty}(\RR)} \lesssim \sqrt{N} \cdot \log^3 N \cdot \left( 1 + \frac{q}{q-2} \right)^2 \|f \|_{L^1(\RR)}.\]
\end{proposition}
Our proof relies on the Multi-Frequency Calder\'{o}n-Zygmund decomposition of \cite{N1}, and indeed is similar to that of \cite[Theorem 5.1]{N1}.
Our modifications were motivated by the outstanding question of \cite{D} concerning the $L^1$-endpoint behavior of the operator
\[ \sup_{k \geq 5} |\Delta_k f|. \]
Indeed, an easy adaptation of our methods yields
that for $1<p\leq 2$
\[ \left\| \sup_{k \geq 5} |\Delta_k f| \right\|_{L^p(\RR)} \lesssim_t N ^{1/p - 1/2} \log^2 N \| f\|_{L^p(\RR)},\]
where the implicit constants depend on a certain $V^t$-quantity in the above notation. This result represents a modest improvement over the current bound of (essentially) $\lesssim_\epsilon N^{1/p -1/2 + \epsilon}$ for $\epsilon > 0$. We refer the reader to \cite[Theorem 1.5]{D} for the precise statement.

\section{The $L^2$ Estimate}
In this section, we work under the frequency-separated hypothesis. The modifications to pass to the general case can be found in \cite[Theorem 4.3]{N1}.
\subsection{Preliminaries}
We shall make use of the following variational results which appear in \cite[\S 3]{B1}.

\begin{lemma}[Lemma 3.30 of \cite{B1}]\label{var}
With $\textbf{f}:=(f_1,f_2,\dots,f_N)$ an $N$-tuple of functions, define pointwise
\[ M_\lambda(x):= M_\lambda( \textbf{f}*\phi_t(x) |t> 0).\]
Then, with $C_\phi:= \| x \phi'(x) \|_1$, and $r>2$, we have
\[
\left\| \sup_{\lambda > 0} \lambda M_\lambda^{1/r}(x) \right\|_{L^2_x(\RR)} \leq C_\phi \frac{1}{r-2} \|\textbf{f} \, \|_H =  C_\phi \frac{1}{r-2} \left\|\left(\sum_{j=1}^N |f_j|^2 \right)^{1/2} \right\|_{L^2_x(\RR)}. \]
\end{lemma}

\begin{lemma}[Lemma 3.33 of \cite{B1}]
Maintaining the notation of the previous lemma, with $N>0$ we have the bound
\[ \left\|\int_0^\infty \min\left\{ M_\lambda^{1/2}(x), N^{1/2} \right\} \ d\lambda \right\|_{L^2_x(\RR)} \lesssim C_\phi (\log N)^2 \left\|\left(\sum_{j=1}^N |f_j|^2 \right)^{1/2} \right\|_{L^2_x(\RR)}.\]
\end{lemma}

To the best of our knowledge, the following technical result has yet to appear in print.

\begin{lemma}\label{Tech}
With the above notation,
\[
\left\| \int_0^\infty \min\left\{ M_\lambda(x)^{1/2}, N^{1/2} M_\lambda(x)^{1/q} \right\} \ d\lambda \right\|_{L^2_x(\RR)}
\lesssim \left( \log N \frac{q}{q-2} \right)^2 \left\| \left(\sum_j |f_j|^2 \right)^{1/2} \right\|_{L^2_x(\RR)}. \]
\end{lemma}
\begin{proof}
Let
\[ F(x):= K_\phi \cdot (\sum_{j=1}^N |\M(f_j)|^2)^{1/2}\]
denote an appropriate amplification of the vector-valued maximal function, where $K_\phi$ depends only on the least radially-decreasing majorant of the Schwartz function $\phi$.

Set
\[ \theta := \frac{1}{\log N}, \ a := \frac{q-2}{2q} \, \theta, \ C := N^{\frac{\theta}{2a}} = N^{\frac{q}{q-2}},\]
and majorize
\[ \aligned
&\int_0^\infty \min\left\{ M_\lambda(x)^{1/2}, N^{1/2} M_\lambda(x)^{1/q} \right\} \ d\lambda \\
& \qquad \qquad \leq
N^{\theta/2} \int_0^{F(x)/C} \lambda^{1-a} M_\lambda(x)^{\frac{1-\theta}{2}+\frac{\theta}{q}} \ \frac{d\lambda}{\lambda^{1-a}} +
N^{\theta/2} \int_{F(x)/C}^{F(x)} \lambda M_\lambda(x)^{\frac{1-\theta}{2}+\frac{\theta}{q}} \ \frac{d\lambda}{\lambda} \\
& \qquad \qquad =: I_1(x) + I_2(x). \endaligned\]
We begin with $I_1(x)$:
\[ \aligned
I_1(x) &\leq
N^{\theta/2} \cdot \left( \sup_{\lambda > 0} \lambda^{1-a} M_\lambda(x)^{\left( \frac{2q}{q-q\theta+2\theta} \right)^{-1}} \right) \cdot
\int_0^{F(x)/C} \ \frac{d\lambda}{\lambda^{1-a}} \\
&\leq N^{\theta/2} \cdot \left( \sup_{\lambda > 0} \lambda^{1-a} M_\lambda(x)^{\left( \frac{2q}{q-q\theta+2\theta} \right)^{-1}} \right) \cdot \frac{1}{a} \left(F(x)/C \right)^a \\
&= \frac{1}{a} \left( \sup_{\lambda > 0} \lambda^{1-a} M_\lambda(x)^{\left( \frac{2q}{q-q\theta+2\theta} \right)^{-1}} \right) \cdot (F(x))^{a} \\
&\leq \frac{1}{a} \left( \sup_{\lambda > 0} \lambda M_\lambda(x)^{\left( \frac{2q(1-a)}{q-q\theta+2\theta} \right)^{-1}} \right) + F(x), \endaligned\]
where we used Young's inequality with exponents $(\frac{1}{a},\frac{1}{1-a})$ in the final line.
With
\[ \aligned
t&:= \frac{2q(1-a)}{q-q\theta+2\theta} = \frac{2q \log N - (q-2)}{q \log N - (q-2)} > 2, \\
\frac{1}{t-2} &= \frac{q-q\theta+2\theta}{2q\theta - 2qa - 4\theta} = \frac{q \log N  - (q-2)}{2(q -2 -\frac{q-2}{2})} \leq \log N \frac{q}{q-2}
\endaligned
\]
and Lemma \ref{var} in mind,
we take $L^2_x$-norms and estimate
\[ \aligned
\| I_1\|_2 &\lesssim
\frac{1}{a} \left( \log N \frac{q}{q-2} \right) \left(\sum_j \|f_j\|_{L^2_x}^2 \right)^{1/2} +
\frac{1}{a} \left(\sum_j \|f_j\|_{L^2_x}^2 \right)^{1/2} \\
&\lesssim \left( \log N \frac{q}{q-2} \right)^2 \left\| \left(\sum_j |f_j|^2 \right)^{1/2} \right\|_{L^2_x}. \endaligned\]

We now turn to the similar, but slightly simpler, second component:
\[ \aligned
I_2(x) &\leq N^{\theta/2} \left( \sup_{\lambda > 0} \lambda M_\lambda(x)^{\left( \frac{2q}{q-q\theta+2\theta} \right)^{-1}} \right) \log C \\
&\lesssim \frac{\theta}{2a} \log N \left( \sup_{\lambda > 0} \lambda M_\lambda(x)^{\left( \frac{2q}{q-q\theta+2\theta} \right)^{-1}} \right) \\
&= \frac{q}{q-2} \log N \left( \sup_{\lambda > 0} \lambda M_\lambda(x)^{\left( \frac{2q}{q-q\theta+2\theta} \right)^{-1}} \right). \endaligned\]
The argument is concluded as above, this time with
\[ \aligned
t&:= \frac{2q}{q-q\theta+2\theta} = \frac{2q \log N}{\log N - (q-2)} > 2, \\
\frac{1}{t-2} &= \frac{q-q\theta+2\theta}{2q\theta -4\theta} =
\frac{q\log N-(q-2)}{2(q-2)} \leq \log N \frac{q}{q-2}; \endaligned\]
we have
\[ \| I_2 \|_2 \lesssim
\left( \log N \frac{q}{q-2} \right)^2 \left\| \left( \sum_j |f_j|^2 \right)^{1/2} \right\|_{L^2_x}. \]
\end{proof}

The previous lemma will be used in conjunction with the lemma below to yield Proposition \ref{main}.

\begin{lemma}\label{comb}
For a sequence
\[ \{ c_k(1),\dots, c_k(N) \}_{k\geq 1} \subset \RR^N,\]
let $M_\lambda = M_\lambda(\{c_k\})$ denote the $l^2([N])$-entropy of the collection, where $l^2([N])$ is the usual $N$-dimensional Hilbert space.
Suppose that $\{\xi_1,\dots,\xi_N\}$ are $1$-separated frequencies.

Then, for $q > 2$,
\[ \aligned
&\left\| \V^q \left( \sum_{j=1}^N c_k(j)e(\xi_j y) \ \big| {k\geq 1} \right) \right\|_{L^2_y[0,1]} \\
& \qquad \qquad \lesssim \int_0^\infty \min \left\{ M_\lambda^{1/2}, N^{1/2} M_\lambda^{1/q} \right\} \ d\lambda +
\int_0^\infty \min\left\{ M_\lambda^{1/2}, N^{1/2} \right\} \ d\lambda. \endaligned \]
\end{lemma}
\begin{proof}
The first term on the right hand side is from the homogeneous $q$-variation, which comes from the proof of \cite[Lemma 3.2]{N1}.
For the second term, which extends the result to non-homogeneous $q$-variation, see the proof of \cite[Lemma 8.4]{D1}.
\end{proof}

We are now ready for the proof of our $L^2$ result.
\subsection{The Proof}
\begin{proposition}
Suppose that $\{\xi_1,\dots,\xi_N\}$ are 1-separated frequencies, and
define
\[ \aligned
D_k (f) &:= \sum_{k=1}^N \int \hat{f}(\xi) \widehat{\phi_{2^k}}(\xi-\xi_j) e(\xi x) \ d\xi  \\
&= \sum_{k=1}^N e(\xi_j x) \int \hat{f}(\xi+\xi_j) \widehat{\phi_{2^k}}(\xi) e(\xi x) \ d\xi. \endaligned\]
Then
\[ \left\| \V^q(D_k (f) \ \big| k \geq 1 ) \right\|_{L^2_x(\RR)} \lesssim (\log N)^2 \cdot \left( 1 + \frac{q}{q-2} \right) \cdot \|f\|_{L^2_x(\RR)}.\]
\end{proposition}
This proposition, and its proof, are very similar to \cite[Proposition 4.1]{N1}, which is in turn is inspired by \cite[Lemma 4.11]{B1}. The only novelty is the use of Lemma \ref{Tech}.
\begin{proof}
Define, via the Fourier transform,
\[ \hat{f_j}(\xi):= \hat{f}(\xi + \xi_j) \psi(\xi),\]
where $1_{[-1/4,1/4]} \leq \hat{\psi} \leq 1_{[-1/2,1/2]}$ (say), and note that by the separation of the $\{ \xi_j\}$ and Plancherel,
\[ \left\| \left( \sum |f_j|^2 \right)^{1/2} \right\|_{L^2_x} \leq \|f \|_2.\]

With this notation in hand, we may express
\[
D_k(f) = \sum_{k=1}^N e(\xi_j x) \int \widehat{f_j}(\xi) \widehat{\phi_{2^k}}(\xi) e(\xi x) \ d\xi = \sum_{k=1}^N e(\xi_j x) f_j * \phi_{2^k}(x); \]
it is therefore enough to prove the vector-valued estimate
\[ \aligned
&\left\| \V^q \left( \sum_{k=1}^N e(\xi_j x) f_j * \phi_{2^k} \ \big| k \geq 1 \right) \right\|_{L^2_x} \\
& \qquad \qquad \lesssim (\log N)^2 \cdot \left( 1 + \frac{q}{q-2} \right)^2 \cdot \left\|\left( \sum_j |f_j|^2\right)^{1/2} \right\|_{L^2_x}. \endaligned\]
\cite[Lemma 4.13]{B1} allows us to reduce the matter to the homogeneous case:
\[
(*) \; \; \; \left\| \tilde{\V}^q \left( \sum_{k=1}^N e(\xi_j x) f_j * \phi_{2^k} \ \big| k \geq 1 \right) \right\|_{L^2_x} \lesssim (\log N)^2 \cdot \left( \frac{q}{q-2} \right)^2 \cdot \left\| \left( \sum_j |f_j|^2\right)^{1/2} \right\|_{L^2_x},\]
where without loss of generality each $f_j$ has frequency support inside $[-1/2,1/2]$; we use Bourgain's averaging argument to do so.

Let $B$ denote the best a priori constant satisfying $(*)$, which we know is finite -- and indeed $\lesssim \sqrt{N}$ by Fefferman-Stein's vector-valued Hardy-Littlewood maximal inequality \cite[\S 2.1]{S}. The job is to prove
\[ B \lesssim (\log N)^2 \cdot \left( \frac{q}{q-2} \right)^2.\]

With $T_y(g)(x):= g(x+y)$, we majorize, for $0 \leq y < \frac{1}{100}$
\[ \aligned
& \left\| \tilde{\V}^q \left( \sum_{k=1}^N e(\xi_j x) f_j * \phi_{2^k} \ \big| k \geq 1 \right) \right\|_{L^2_x} \\
& \qquad \leq \left\| \tilde{\V}^q \left( \sum_{k=1}^N e(\xi_j x) (T_yf_j) * \phi_{2^k} \ \big| k \geq 1 \right) \right\|_{L^2_x} +
\left\| \tilde{\V}^q \left( \sum_{k=1}^N e(\xi_j x) (f_j - T_yf_j) * \phi_{2^k} \ \big| k \geq 1 \right) \right\|_{L^2_x} \\
& \qquad =: S_1(y) + B \left\|\left( \sum_j |(f_j-T_yf_j)|^2\right)^{1/2} \right\|_{L^2_x} \\
& \qquad < S_1(y) + \frac{B}{2} \left\| \left( \sum_j |f_j|^2\right)^{1/2} \right\|_{L^2_x}, \endaligned\]
where we used the band-limited nature of the $\{f_j\}$ in the last inequality.

Averaging this inequality over $0 \leq y < \frac{1}{100}$, we see that we just need to prove
\[ \aligned
\left( \int_0^{\frac{1}{100}} |S_1(y)|^2 \ dy \right)^{1/2} &\leq
\left\| \left\| \tilde{\V}^q \left( \sum_{k=1}^N e(\xi_j y) \cdot \left( e(\xi_j x) (f_j) * \phi_{2^k}(x) \right) \ \big| k \geq 1 \right) \right\|_{L^2_y[0,1]} \right\|_{L^2_x} \\
&\lesssim (\log N)^2 \cdot \left( \frac{q}{q-2} \right)^2 \cdot \left\| \left( \sum_j |f_j|^2\right)^{1/2} \right\|_{L^2_x}. \endaligned \]
If we apply Lemma \ref{comb} to the inner integral, we get the bound
\[ \aligned
&\left\| \tilde{\V}^q \left( \sum_{k=1}^N e(\xi_j y) \cdot \left( e(\xi_j x) (f_j) * \phi_{2^k}(x) \right) \ \big| k \geq 1 \right) \right\|_{L^2_y[0,1]}  \\
& \qquad \leq \int_0^\infty \min\{ M_\lambda^{1/2}(x), N^{1/2} M_\lambda^{1/q}(x) \} \ d\lambda, \endaligned\]
where $M_\lambda(x)$ is the (pointwise) $l^2[N]$-$\lambda$-entropy of
\[ \big( e(\xi_1 x) (f_1) * \phi_{2^k}(x), e(\xi_2 x) (f_2) * \phi_{2^k}(x), \dots, e(\xi_N x) (f_N) * \phi_{2^k}(x) \big)_{k \geq 1}.\]
Taking $L^2_x$-norms and using Lemma \ref{Tech} now yields the result.
\end{proof}

\section{The Weak-Type $(1-1)$ Estimate}
As announced, in this section we prove Proposition \ref{1-1}, restated below for the reader's convenience:
\begin{proposition}
With $D_k(f) := \sum_{|\omega|= 2^{-k} : \omega \cap \Sigma \neq \emptyset } \int \hat{f}(\xi) \hat{\phi_\omega}(\xi) e(\xi x) \ d\xi,$
\[ \left\| \V^q(D_k (f)) \right\|_{L^{1,\infty}} \lesssim \sqrt{N} \cdot \log^3 N \cdot \left( 1 + \frac{q}{q-2} \right)^2 \|f \|_{L^1}.\]
\end{proposition}

Before turning to the proof proper, we briefly recall the multi-frequency Calder\'{o}n-Zygmund Decomposition:
\subsection{The Multi-Frequency Calder\'{o}n-Zygmund}
\begin{theorem}[\cite{N1}, Theorem 1.1]\label{MFCZ}
For any $f \in L^1(\RR)$, $\lambda >0$, there exists a decomposition
\[ f = g + b = g + \sum_{J} b_J\]
for a disjoint collection of (dyadic) intervals $\{J\}$ according to the following properties:
\begin{enumerate}
\item $\|f_J\|_1 := \|f1_J\|_1 \lesssim \frac{\lambda|J|}{\sqrt{N}}$;
\item $\| g_J \|_2 := \| f_J - b_J \|_2 \lesssim |J|^{1/2} \lambda$;
\item $\|g\|_2^2 \lesssim \sqrt{N} \lambda \|f\|_1$;
\item $\supp b_J \subset 3J, \ \sum_{J} |J| \lesssim \sqrt{N}\frac{\|f\|_1}{\lambda}$;
\item $\|b_J\|_1 \lesssim \lambda |J|$; and
\item $b$ is orthogonal to the frequencies: $\int b_J(x) e(- \xi_j x) = 0$ for each $J, n$.
\end{enumerate}
\end{theorem}

It will be convenient to discretize our operator. We do so as follows:
\subsection{A Discretization}

With $\phi_\omega$ as above, 
let $\eta \in \s(\RR)$ be a positive mean-one mollifier with $\supp \eta \subset [-0.1,0.1]$ (say), and let $A$ be a Schwartz function with
\[ 1_{[-1.4,1.4]} \leq \hat{A}(\xi) \leq 1_{[-1.6,1.6]}.\]
%

With $B:= \eta^{\vee}$, we use windowed Fourier series as in \cite[\S 6]{D} to express
\[
\hat{f}\widehat{\phi_\omega}(\xi) = \frac{1}{|\omega|} \sum_{l \in \Z} \left( \int \hat{f}(t)\widehat{\phi_\omega}(t) e( \frac{l}{|\omega|} t) \ dt \right) e(- \frac{l}{|\omega|} \xi) \widehat{(AB)}\left( \frac{\xi-\xi_\omega}{|\omega|} \right).\]

Set
\[ \phi_{l,\omega}(z):= e(-z \xi_\omega) \frac{1}{|I|} \tilde{\phi} \left( \frac{z - l|I|}{|I|} \right) \]
where $\tilde{g}(z):= g(-z)$ denotes reflection about the origin, and $|I| = |\omega|^{-1}$; we note that in the special case $l=0$, we have
\[ \phi_{l,\omega}(z) = e(-z \xi_\omega) \tilde{\phi}_{|I|}(z).\]

After an argument with the Fourier transform, we have
\[
f*\phi_\omega(x) = \sum_{l \in \Z} \langle f, \phi_{l,\omega} \rangle e(\xi_\omega x) \cdot (AB)\left(\frac{x- l|I|}{|I|} \right) =:\sum_{l \in \Z} \langle f, \phi_{l,\omega} \rangle e(\xi_\omega x) \cdot T_{l,|I|} \big(AB\big)(x).
\]
Consequently, with $2^k = |I| = |\omega|^{-1}$, we may express
\[
D_k(f):= \sum_{l \in \Z } \left( \sum_{|\omega|=2^{-k} : \omega \cap \Sigma \neq \emptyset} \langle f, \phi_{l,\omega} \rangle e(\xi_\omega x) \right) T_{l,|I|}\big(AB\big)(x),\]
as a sum of projections.

\begin{proof}[Proof of Theorem \ref{main}]
By homogeneity, it suffices to prove the weak-type estimate at height $\lambda = 1$. We therefore apply the Multi-Frequency Calder\'{o}n-Zygmund decomposition \ref{MFCZ} at height 1. By standard arguments, it suffices to show that for each $J$ selected
\[ \| \V^q(D_k (b_J )) \|_{L^1((J^*)^c)} \lesssim |J|,\]
where $J^* := 100J$. By scale and translation invariance, we may assume $J = [0,1)$; we proceed to majorize
\[ \V^q(D_k (b_J )) \leq \sum_{k \geq 0} |D_k(b_J)| + \sum_{k < 0} |D_k(b_J)| ,\]
according to scale, and estimate each term separately in $\| - \|_{L^1((J^*)^c)}$.

We begin with the more demanding
\begin{proof}[Case 1: $k \geq 0$]
We use our discretization to bound
\[ \aligned
&\left\| \sum_{k \geq 0} |D_k(b_J)| \right\|_{L^1} \\
& \qquad
\leq \sum_{l \in \Z} \sum_{|I|=2^k, k \geq 0} \left\| \left(\sum_{|\omega|=2^{-k} : \omega \cap \Sigma \neq \emptyset} \langle b_J, \phi_{l,\omega} \rangle e(\xi_\omega x) T_{l,|I|}A \right) \cdot \left( T_{l,|I|}B \right) \right\|_{L^1} \\
& \qquad
\leq \sum_{l \in \Z} \sum_{|I|=2^k, k \geq 0} \| T_{l,|I|}B\|_{L^2} \left\| \sum_{|\omega|=2^{-k} : \omega \cap \Sigma \neq \emptyset} \langle b_J, \phi_{l,\omega} \rangle e(\xi_\omega x) T_{l,|I|}A \right\|_{L^2} \\
& \qquad
\lesssim \sum_{l \in \Z} \sum_{|I|=2^k, k \geq 0} |I|^{1/2} \left\| \sum_{|\omega|=2^{-k} : \omega \cap \Sigma \neq \emptyset} \langle b_J, \phi_{l,\omega} \rangle e(\xi_\omega x) T_{l,|I|}A(x) \right\|_{L^2}. \endaligned \]
We now split our sum
\[
\sum_{|\omega|=2^{-k} : \omega \cap \Sigma \neq \emptyset} := \sum_{i=1}^{10} \sum_{\omega,i} := \sum_{i=1}^{10} \left( \sum_{\omega \cap \Sigma \neq \emptyset, \ \omega = 2^{-k}[m,m+1), \ m \equiv i \mod 10} \right), \]
so that we may bound
\[ \aligned
\left\| \sum_{|\omega|=2^{-k} : \omega \cap \Sigma \neq \emptyset} \langle b_J, \phi_{l,\omega} \rangle e(\xi_\omega x) T_{l,|I|}A(x) \right\|_{L^2}
&\leq
\sum_{i=1}^{10} \left\| \sum_{\omega,i} \langle b_J, \phi_{l,\omega} \rangle e(\xi_\omega x) T_{l,|I|}A(x) \right\|_{L^2} \\
&= \sum_{i=1}^{10} \left\| \left( \sum_{\omega,i} |\langle b_J, \phi_{l,\omega} \rangle |^2 | T_{l,|I|}A |^2 \right)^{1/2} \right\|_{L^2}, \endaligned \]
taking into account the band-limited nature of our function $A$.
Our bounds will be uniform in $1 \leq i \leq 10$, so we suppress $i$-dependence throughout.

We begin with the top term, $l=0$, and suppress $l$-dependence as well;
we will show that for each $|I|$,
\[ \left\| \left( \sum_{\omega}  |\langle b_J, \phi_{\omega} \rangle|^2 |T_{|I|}A|^2 \right)^{1/2} \right\|_{L^2} \lesssim \frac{1}{|I|^{3/4}},\]
since a sum over $|I| \geq 1$ will yield the $l=0$ bound
\[
\sum_{k \geq 0} \left\| \sum_{\omega } \langle b_J, \phi_{l,\omega} \rangle e(\xi_\omega x) T_{l,|I|}\big( AB \big) \right\|_{L^1} \lesssim 1.\]

Regarding our scale $|I|=2^k$ as fixed, we abbreviate
\[ \rho(z):= \tilde{\phi}_{|I|}(z) - \tilde{\phi}_{|I|}(0) =
\frac{1}{|I|} \tilde{\phi}(\frac{z}{|I|}) - \frac{1}{|I|} \tilde{\phi}(0),\]
and use the orthogonality of the $\{b_J\}$ to our frequencies $\{ e( -z \xi_\omega)\}$ to re-express the integrand as
\[
\left( \sum_{\omega} | \langle b_J, \rho(z) e( -z \xi_\omega) \rangle|^2 |T_{|I|}A|^2(x) \right)^{1/2}; \]
we remark the mean-value theorem yields the bound
\[ (*_0) \; \; \; | \rho(z) | \lesssim \frac{1}{|I|^2} \]
pointwise near $3J$, with implicit constant depending on $\| \phi' \|_{L^{\infty}}$.

Now, using the triangle inequality, it suffices to estimate the $L^2$ norm of each of the following functions separately:
\[ \aligned
F(x) &:= \left( \sum_{\omega} | \langle f_J,\rho(z) e( -z \xi_\omega) \rangle|^2 |T_{|I|}A|^2(x) \right)^{1/2} \\
G(x) &:= \left( \sum_{\omega} | \langle g_J, \rho(z) e( -z \xi_\omega) \rangle|^2 |T_{|I|}A|^2(x) \right)^{1/2}.
\endaligned \]
$F$ is easy, since we have the pointwise bound
\[
F(x) \leq
\left( \sum_\omega  \left( \|f_J \|_1 \cdot \frac{1}{|I|^2} \right)^2 \right)^{1/2} |T_{|I|}A|(x) \leq \sqrt{N} \|f_J\|_1 \frac{1}{|I|^2} |T_{|I|}A|(x). \]
Consequently,
\[ \| F\|_{L^2} \lesssim \sqrt{N} \|f_J\|_1 \frac{1}{|I|^2} |I|^{1/2} \lesssim \frac{1}{|I|^{3/2}}.\]

We use duality to handle $\| G\|_{L^2}$. Specifically, for an appropriate $\sum \| \psi_\omega \|_{L^2}^2 \leq 1$,
we estimate
\[ \aligned
\| G\|_2 &\lesssim
\left|\int \sum_{\omega} \langle \psi_\omega , T_{|I|}A \rangle \cdot \langle g_J , \rho(z) e( -z \xi_\omega) \rangle \right| \\
&= \left| \langle g_J , \sum_{\omega} \rho(z) e( -z \xi_\omega) \cdot \langle \psi_\omega , T_{|I|}A \rangle \rangle \right| \\
&\lesssim \left\| \sum_{\omega} \rho(z) e( -z \xi_\omega) \cdot \langle \psi_\omega , T_{|I|}A \rangle  \right\|_{L^2(3J)} \\
&\leq
\left\| \sum_{\omega} \chi(z) \rho(z) e( -z \xi_\omega) \cdot \langle \psi_\omega , T_{|I|}A \rangle \right\|_{L^2(\RR)},
\endaligned \]
where $1_{[-4,4]} \leq \chi \leq 1_{[-5,5]}$ (say) is a smooth bump function, and we have made use of the normalization $\| g_J \|_{L^2} \lesssim 1$, along with the support condition $\supp g_J \subset 3J$.

We will bound the above expression by the square root of
\[ \aligned
& \sum_\omega \|  \rho(z) e( -z \xi_\omega) \chi \|_{L^2}^2 \cdot |\langle \psi_\omega,  T_{|I|}A \rangle|^2
+
\sum_{\omega \neq \omega'}
|\langle \psi_\omega , T_{|I|}A \rangle| K(\omega,\omega') |\langle \psi_{\omega'} , T_{|I|}A \rangle| \\
& \qquad \lesssim |I|
\left( \sum_\omega \|\rho(z) \chi(z) e( -z \xi_\omega) \|_{L^2}^2 \|\psi_\omega\|_2^2 +
\sum_{\omega \neq \omega'}
 \| \psi_\omega \|_{L^2} K(\omega,\omega')  \| \psi_{\omega'} \|_{L^2} \right) \\
& \qquad \qquad =: |I| \left( \mathcal{D} + \mathcal{O} \right), \endaligned \]
where
\[ K(\omega,\omega'):=
\left| \int |\rho(z) \chi(z) |^2 e( (\xi_{\omega'} - \xi_{\omega}) z ) \ dz \right|.\]
Using $(*)$, the diagonal estimate is straightforward:
\[ \mathcal{D} \leq
\max_\omega
\| \rho(z) \chi(z) e(-\xi_\omega z) \|_{L^2}^2 \\
\lesssim \frac{1}{|I|^4}. \]
We will use the bound
\[ K(\omega,\omega') \lesssim \min \left\{ \frac{1}{|\xi_\omega - \xi_{\omega'}|}, \frac{1}{|\xi_\omega - \xi_{\omega'}|^2} \right\} \frac{1}{|I|^4} \leq \frac{1}{|\xi_\omega - \xi_{\omega'}|^{3/2}} \cdot \frac{1}{|I|^4} \]
to control the off-diagonal term. We obtain this estimate by integrating by parts once and twice, and by using that on $\supp \chi$ we have the bounds
\[ |\rho| + |\rho'| \lesssim \frac{1}{|I|^2}, \ |\rho''| \lesssim \frac{1}{|I|^3}. \]
Using the fact that for $\omega \neq \omega'$
\[|\xi_\omega - \xi_{\omega'}| > |\omega| = |I|^{-1}, \]
we may bound, for each $\omega$,
\[
\sum_{\omega' \neq \omega} K(\omega, \omega') \leq \frac{1}{|I|^4}\sum_{\omega' \neq \omega} \frac{1}{|\xi_\omega - \xi_{\omega'}|^{3/2}}
\lesssim \frac{|I|^{3/2}}{|I|^4} \sum_{n=1}^N \frac{1}{n^{3/2}} \lesssim \frac{1}{|I|^{5/2}}. \]


With this in hand, Young's inequality leads to the bound
\[
\mathcal{O} \lesssim \sum_{\omega \neq \omega'} \|\psi_\omega \|_{L^2}^2 |K(\omega,\omega')|
 \lesssim \sum_\omega \| \psi_\omega\|_{L^2}^2 \frac{1}{|I|^{5/2}} \leq
 \frac{1}{|I|^{5/2}} , \]

Putting things together, we have obtained the estimate
\[ |I| \left( \mathcal{D} + \mathcal{O} \right) \lesssim \frac{1}{|I|^{3/2}};\]
taking a square-root leads to the bound
\[ \| G  \|_{L^2} \lesssim
\frac{1}{|I|^{3/4}}. \]
Combining all our estimates, we may bound
\[ \left\| \left( \sum_{\omega} |\langle b_J, \phi_{\omega} \rangle|^2 |T_{l,|I|}A|^2 \right)^{1/2} \right\|_{L^2} \leq \| F\|_{L^2} + \|G\|_{L^2} \lesssim \frac{1}{|I|^{3/4}},\]
which concludes the $l=0$ case.

The remaining terms $l \neq 0$ are handled similarly, using the analogous mean-value estimate on $J$
\[ (*_{|l|}) \; \; \; \left| \frac{1}{|I|}\tilde{\phi}\left( \frac{z - l|I|}{|I|}  \right) -
\frac{1}{|I|}\tilde{\phi}\left( \frac{- l|I|}{|I|}  \right) \right| \lesssim \frac{1}{|I|^2} |l|^{-2}, \]
and similarly
\[ | \partial_z^{\alpha} \left( \frac{1}{|I|} \tilde{\phi}\left(\frac{z-l|I|}{|I|}\right) \right) | \lesssim_\phi \frac{l^{-2}}{|I|^{1+\alpha}}, \ \alpha = 1,2,\]
where the implicit constants depend further on the decay of $\phi', \phi''$.

This concludes the first case, $|I| \geq |J|=1$.
\end{proof}

We next turn to
\begin{proof}[Case 2: $k <0 $]
With $J = [0,1)$ as above, and regarding the scale $|I|= 2^{k}$ as fixed, we use the decay of the functions $\phi_\omega$ to estimate
\[ \aligned
\left\| \sum_{|\omega|=2^{-k} : \omega \cap \Sigma \neq \emptyset} \phi_\omega * b_J \right\|_{L^1((J^*)^c)} &\leq N \max_{\omega} \| \phi_\omega * b_J \|_{L^1((J^*)^c)} \\
&\leq N \|b_J\|_1 \left\| \sup_{y \in 3J} | \phi_\omega(z - y)| \right\|_{L^1((J^*)^c)} \\
&\lesssim N |I|^{M-1} (e)^{1-M}, \endaligned \]
where we may take
\[ M=M(N) = \log N, \]
and the implicit constant depends on the best bound, $D_M$, satisfying
$|\phi(z)| \leq D_M |z|^{-M}$ as $|z| \to \infty$.
Summing over $|I|$ concludes the proof.
\end{proof}
\end{proof}
\begin{remark}
As mentioned above, this argument is easily adapted to handle the (simpler) operator
\[ Vf(x):= \left( \sum_{n=1}^N \V^r_{k \geq 5} (1_I(x) \langle f, \phi_{I,n} \rangle )^2 \right)^{1/2} \]
of \cite[\S 5]{D}. Indeed, one can similarly prove that
\[ \| Vf\|_{L^{1,\infty}} \lesssim
\sqrt{N}  \|f\|_1,\]
with implicit constants depending on various (scalar) variational quantities discussed in \cite{D}.
\end{remark}

\begin{remark}
This proof further generalizes to the case of multiple weights, i.e.\ the $\{D_k\}$ are replaced with $\{ \Delta_k \}$ of \cite{N1}.
\end{remark}

\section{Multipliers with Bounded $r$-Variation}
This section will be devoted to Theorem \ref{rvar}.

We will use the following decomposition lemma \cite{CRS}:

\begin{lemma}
Suppose $1 \leq r < \infty$, and $\| g \|_{\V^r(\RR)} < \infty$. Then there exist collections of intervals $\mathcal{I}_j$ so that we may decompose
\[ g = \sum_{j \geq 0} \sum_{I \in \mathcal{I}_j} d_I 1_I,\]
where $|\mathcal{I}_j| \leq 2^j$, $|d_I| \leq 2^{-j/r} \| g\|_{\V^r(\RR)}$ for each $I \in \mathcal{I}_j$, and the intervals $I \in \mathcal{I}_j$ are pairwise disjoint.
\end{lemma}

Using this lemma, it is enough to establish the following:
\begin{proposition}
\[ \left\| \left( \sum_{\omega \in X} d_\omega 1_\omega \hat{f} \right)^{\vee} \right\|_{L^q(\RR)}
\lesssim \sup |d_\omega| |X|^{|\frac{1}{q} - \frac{1}{2}|} \| f\|_{L^q(\RR)}; \]
\end{proposition}
by duality, it is enough to consider the $1 < q \leq 2$ case. We will establish the weak-type estimate

\begin{proposition}\label{weak}
\[ \left\| \left( \sum_{\omega \in X} d_\omega 1_\omega \hat{f} \right)^{\vee} \right\|_{L^{1,\infty}(\RR)}
\lesssim \sup |d_\omega| |X|^{1/2} \| f\|_{L^1(\RR)}. \]
\end{proposition}

So, viewing $\{ \omega \}$ and $\{ d_\omega \}$ as fixed, without loss of generality $|d_\omega | \leq 1$, we let
\[ Tf = T_{\{ \omega \}, \{ d_\omega \} } := \left( \sum_{\omega \in X} d_\omega 1_\omega \hat{f} \right)^{\vee} \]
denote the relevant operator.

The proof of Proposition \ref{weak} will follow along similarly to that of Proposition \ref{1-1}; technical complications arise from the fact that our multipliers are given by (rough) indicator functions, rather than by (smooth) Schwartz weights.

We therefore begin with a smooth decomposition of each $1_\omega$.

Let $\{ u(\omega) \}$ be (dyadic) Whitney intervals so that for each $\omega$
\begin{itemize}
\item $100 u(\omega) \subset \omega$;
\item $\bigcup u(\omega) = \omega$ is an almost disjoint union;
\item there exists an absolute $K= O(1)$ so that
$\sum_{u(\omega)} 1_{20u(\omega)} \leq K$, i.e.\ the collection $\{ 20 u(\omega) \}$ has bounded overlap.
\end{itemize}

Associated to this decomposition, we define two families of smooth functions,
\[ \{ \phi_{u(\omega)} \} \ \text{ and } \{A_{u(\omega)} \}. \]

The $\{\phi_{u(\omega)} \}$ satisfy
\[ 1_\omega = \sum_{u(\omega)} \widehat{\phi}_{u(\omega)}, \ \text{ and }\supp \widehat{\phi}_{u(\omega)} = I(u(\omega)),\]
 where $I(u(\omega))$ is an interval concentric with $u(\omega)$ of without loss of generality dyadic length and $2 < \frac{|I(u(\omega))|}{|u(\omega)|}  \leq 4$.
By our Whitney decomposition, we may assume that there exists an absolute $R = O(1)$ so that for each $\omega$ and each $j$,
\[ \# \{ u(\omega) : I(u(\omega)) = 2^{-j} \} \leq R.\]

We similarly define $\widehat{A}_{u(\omega)}$ to be smooth functions supported in $15 u(\omega)$ and identically equal to one on $10u(\omega)$.

We shall require uniform decay on $\{ \phi_{u(\omega)} \}$ and $\{ A_{u(\omega)} \}$: if $I(u(\omega)) = 2^{-j}$, then we demand that
\[ \left\| \widehat{\phi}_{u(\omega)}^{(M)} \right\|_\infty + \left\| \widehat{A}_{u(\omega)}^{(M)} \right\|_\infty \leq C_M 2^{Mj} \]
for $M = \log_2 N $.

Finally, we consider a family of smooth, positive, mean-one mollifiers,
\[ \int \eta_j = 1, \ \supp \eta \subset \left\{ \xi: |\xi| \leq \frac{1}{1000} 2^{-j} \right\},\]
so that for $|I(u(\omega))| = 2^{-j}$,
\[ \eta_j * \hat{A}_{u(\omega)} = 1 \]
on $5I(u(\omega))$ (say).

Next, if $c_\omega$ denotes the center of each interval $\omega$, we define rescaled, shifted-to-the-origin versions of our multipliers $\left\{ \phi_{u(\omega)}, A_{u(\omega)} \right\}$:

For $|I(u(\omega))| = 2^{-j}$, we define
\[
\hat{\phi}_{u(\omega)*}(2^j(\xi - c_\omega)) := \hat{\phi}_{u(\omega)}(\xi)
\text{ i.e. }
\hat{\phi}_{u(\omega)*}(\xi) := \hat{\phi}_{u(\omega)}(\frac{\xi}{2^j} + c_\omega),\]
and similarly for $A_{u(\omega)*}$.

Collect
\[ \F_j:= \left\{ u(\omega) : |I(u(\omega))| = 2^{-j} \right\} \]
and sparsify
\[ \F_j = \bigcup_{i=1}^{RK} \F_j^k \]
into $RK$ families, so that
\begin{itemize}
\item Each $\F_j^k$ has at most one $u(\omega)$ with $|I(u(\omega))|= 2^{-j}$ from each $\omega$ (i.e.\ if $u(\omega), u'(\omega') \in \F_j^k$ with $|I(u(\omega))| = |I(u'(\omega'))|$, then necessarily $\omega \neq \omega'$;
\item For each $u(\omega) \in \F_j^k$, the supports of each $\hat{A}_{u(\omega)}$ are disjoint.
\end{itemize}

Since our estimates will be uniform in $k$, in what follows, we shall suppress all dependence of $k$ in our $\F_j^k$.

With this in mind, we decompose
\[ \widehat{Tf} := \sum_j \widehat{T_j f} := \sum_j \left( \sum_{\F_j} d_{u(\omega)} \hat{\phi}_{u(\omega)} \hat{f} \right),\]
where we denote $d_{u(\omega)} = d_\omega$ for ${u(\omega)} \subset \omega$.

Following along the lines of the previous section we discretize each $T_j$:

With $|I({u(\omega)})| = 2^{-j}$ we used windowed Fourier series to express
\[ \hat{\phi}_{u(\omega)}(\xi) \hat{f}(\xi) = 2^j \sum_{l \in \Z} \left( \int \hat{f}(t) \hat{\phi}_{u(\omega)}(t) e( 2^j l t) \ dt \right) e( - 2^j l \xi) \hat{A}_{u(\omega)}*\eta_j(\xi). \]
After an argument with the Fourier transform, we have
\[ \aligned
& d_{u(\omega)} \phi_{u(\omega)} * f(x) \\
& \qquad = \sum_{l \in \Z} \langle f(x), e(-c_\omega x) 2^{-j} \phi_{u(\omega)*}(2^{-j}x - l) \rangle \cdot \big( e( c_\omega x) A_{u(\omega)*}(2^{-j} x - l) \big) \left( d_{u(\omega)} \cdot \eta_j^{\vee}(x - 2^j l) \right) \\
& \qquad =:
\sum_{l \in \Z} \langle f(x), e(-c_\omega x) 2^{-j} \phi_{u(\omega)*}(2^{-j}x - l) \rangle \cdot \big( E_{u(\omega),j,l}(x) \big) \left( d_{u(\omega)} \cdot \eta_j^{\vee}(x - 2^j l) \right)
; \endaligned\]
for convenience, we remark that
\[ \aligned
\widehat{E_{u(\omega),j,l}}(\xi) &:= \widehat{ \left( e(c_\omega x) A_{u(\omega)*} (2^{-j}x - l) \right) } ( \xi ) \\
& = e(-\xi 2^j l ) e(c_w 2^j l) \hat{A}_{u(\omega)}(\xi). \endaligned
\]

We now apply the Multi-Frequency Calder\'{o}n-Zygmund decomposition at height $\lambda = 1$, and extract $b = \sum_J b_J$ orthogonal to $\{ e( -c_\omega x ) : \omega \in X\}$.
By standard arguments, and scale and translation invariance, our task is once again to show that
\[ \| T b_J \|_{L^1( \{ |x| \geq 100 \} )} \lesssim 1 \]
where we may assume that $J = [0,1)$.

We will majorize
\[ \| T b_J \|_{L^1( \{ |x| \geq 100 \} )} \leq \sum_j \| T_j b_J \|_{L^1( |x| \geq 100 )} \]
according to scale and estimate each term separately.

We turn to
\subsection{The Proof of Theorem \ref{rvar}}
As mentioned, the strategy here is very similar to that of Proposition \ref{1-1}.

We begin with the more involved case, $j \geq 0$
\begin{proof}[The Proof, $j \geq 0$ ]
We further decompose
\[ \aligned
T_j b_J &:= \sum_l T_{j,l} b_J \\
&:= \sum_l \left( \sum_{\F_j} \langle b_J(x), e(-c_\omega x) 2^{-j} \phi_{u(\omega)*}(2^{-j}x - l) \rangle \cdot \big( e(c_\omega x) A_{u(\omega)*}(2^{-j}x-l) \big) \right) \cdot \left(d_{u(\omega)} \eta_j(x - 2^j l) \right) \\
&= \sum_l \left( \sum_{\F_j} \langle b_J(x), e(-c_\omega x) 2^{-j} \phi_{u(\omega)*}(2^{-j}x - l) \rangle \cdot E_{u(\omega),j,l}(x) \right) \cdot \left(d_{u(\omega)} \eta_j(x - 2^j l) \right)
; \endaligned \]
we will estimate each $\| T_{j,l} b_J \|_1$, and sum over $(j,l)$; by the arguing as above, and using the \emph{uniform} decay of the $\{ \phi_{u(\omega)*} \}$, it will be enough to estimate the top $l=0$ term. We will therefore suppress $l$-dependence in what follows.

We begin by applying Cauchy-Schwarz to estimate
\[  \| T_j b_J \|_1 \lesssim 2^{j/2} \left\|
\sum_{\F_j} \langle b_J(x), e(-c_\omega x) 2^{-j} \phi_{u(\omega)*}(2^{-j}x) \rangle E_{u(\omega),j}(x) \right\|_2,\]
then use the orthogonality of
\[ \left\{ E_{u(\omega),j}(x) \right\} \]
(see above) to express
\[ \aligned
& \left\| \sum_{\F_j} \langle b_J(x), e(-c_\omega x) 2^{-j} \phi_{u(\omega)*}(2^{-j}x) \rangle E_{u(\omega),j}(x) \right\|_2 \\
& \qquad =
\left\| \left( \sum_{\F_j} |\langle b_J(x), e(-c_\omega x) 2^{-j} \phi_{u(\omega)*}(2^{-j}x) \rangle |^2 |E_{u(\omega),j}(x)|^2 \right)^{1/2} \right\|_2\\
& \qquad =
\left\| \left( \sum_{\F_j} |\langle b_J(x), e(-c_\omega x) 2^{-j} \phi_{u(\omega)*}(2^{-j}x) \rangle |^2 |A_{u(\omega)*}(2^{-j}x)|^2 \right)^{1/2} \right\|_2. \endaligned \]
Since $b_J$ is orthogonal to each $e(c_\omega x)$, we may replace the inner products
\[ \aligned
\langle b_J(x), e(-c_\omega x) 2^{-j} \phi_{u(\omega)*}(2^{-j}x) \rangle &\equiv
\langle b_J(x), e(-c_\omega x) \left( 2^{-j} \phi_{u(\omega)*}(2^{-j}x) -
2^{-j} \phi_{u(\omega)*}( 0 ) \right) \rangle \\
&=: \langle b_J(x), e(-c_\omega x) \rho_{u(\omega),j}(x) \rangle.
\endaligned \]
We observe that for $|x| \leq 5$ (say) we have the $\rho_{u(\omega),j}$-\emph{uniform} bounds
\[ \aligned
|\rho_{u(\omega),j}| + |\rho_{u(\omega),j}'| &\lesssim 2^{-2j} \\
|\rho_{u(\omega),j}''| &\lesssim 2^{-3j} , \endaligned\]
the first point following by the mean-value theorem.

We now majorize
\[ \aligned
& \left( \sum_{\F_j} |\langle b_J(x), e(-c_\omega x) \rho_{u(\omega),j}(x) \rangle |^2 |A_{u(\omega)*}(2^{-j}x)|^2 \right)^{1/2} \\
& \qquad \leq
\left( \sum_{\F_j} |\langle f_J(x), e(-c_\omega x) \rho_{u(\omega),j}(x) \rangle |^2 |A_{u(\omega)*}(2^{-j}x)|^2 \right)^{1/2} \\
& \qquad \qquad +
\left( \sum_{\F_j} |\langle g_J(x), e(-c_\omega x) \rho_{u(\omega),j,l}(x) \rangle |^2 |A_{u(\omega)*}(2^{-j}x)|^2 \right)^{1/2} \\
& \qquad =: F(x) + G(x), \endaligned \]
and estimate each function separately in $L^2$.

Using our uniformity conditions on the $\{\rho_{u(\omega),j} \}, \{ A_{u(\omega)} \}$, we now find ourselves in the situation of Proposition \ref{1-1}, and we may argue as above to conclude
\begin{itemize}
\item $\| F \|_2 \lesssim 2^{ - 3j/2 }$; and
\item $\| G\|_2 \lesssim 2^{-3j/4}$.
\end{itemize}

Consequently, we have achieved the upper bound on
\[ \| T_j b_J \|_1 \lesssim 2^{j/2} \cdot 2^{-3j/4}  \lesssim 2^{-j/4}.\]
We may similarly deduce
\[ \| T_{j,l} b_J \|_1 \lesssim 2^{-j/4} \cdot \left( |l| + 1 \right)^{-2}; \]
since this is summable over $l \in \Z, j \geq 0$, we have completed this part of the proof.
\end{proof}

We next turn to the simpler $j < 0$ case, where we rely solely on the decay of our functions.
\begin{proof}[The Proof, $j < 0$]
Using our (uniform) decay estimates on $\phi_{u(\omega)}$,
we simply estimate, with $M = \log_2N$
\[ \aligned
\| T_j b_J \|_{L^1( (100J)^* )} &\leq N \max_{ {u(\omega)} \in \F_j } \left\| \phi_{u(\omega)} * b_J \right\|_{L^1( (100J)^c )} \\
&\lesssim N \| b_J \|_1 \left\| 2^{-j} \min\left\{ 1, \left( \frac{2^j}{|x|} \right)^M \right\}\right\|_{L^1(|x| \geq 50)} \\
&\lesssim N (2^j)^{M-1}. \endaligned \]
Summing the foregoing in $j$ yields an upper estimate of
\[ N 2^{1-M} \lesssim 1 \]
for $M= \log_2N$.
\end{proof}

\begin{remark}
An interesting question, raised in \cite{RO}, concerns the $L^p$-behavior of Bourgain's \emph{maximal} (rough) singular integral $\S 1$:
\begin{problem}
For any collection of $N$ frequencies, $\Sigma := \{\xi_1,\dots, \xi_N \} \subset \RR$, consider the maximal operator
\[ \M^*f:= \sup_j | \left( \hat{f} 1_{R_j} \right)^{\vee}|(x), \]
where $R_j$ is the $2^{-j}$ neighborhood of $\Sigma$.
For $p \neq 2$, what are the best $N$-dependent constants, $C_p(N)$, so that
\[ \| \M^*f \|_{L^p(\RR)} \leq C_p(N) \|f\|_{L^p(\RR)}? \]
\end{problem}
Although the operator $\M^*f$ is too rough to be handled by present technique, we look forward to pursuing this line of inquiry in future work.
\end{remark}

\end{document}